\documentclass[12pt,leqno,a4paper,reqno]{amsart}
\usepackage{amssymb, enumerate}                  

\usepackage[usenames]{color}    
\usepackage{hyperref}   
\usepackage{enumitem}

\hypersetup{pdftex,                           
bookmarks=true,
pdffitwindow=true,
colorlinks=true,
citecolor=black,
filecolor=black,
linkcolor=black,
urlcolor=black,
hypertexnames=true}

\DeclareMathOperator{\Hom}{Hom}
\DeclareMathOperator{\End}{End}
\DeclareMathOperator{\Aut}{Aut}
\DeclareMathOperator{\Inf}{Inf}
\DeclareMathOperator{\Res}{Res}
\DeclareMathOperator{\Syl}{Syl}    
\DeclareMathOperator{\proj}{(proj)}
\DeclareMathOperator{\Image}{Im}
\DeclareMathOperator{\diag}{diag}

\newcommand{\GL}{{\operatorname{GL}}}
\newcommand{\GU}{{\operatorname{GU}}}
\newcommand{\PGL}{{\operatorname{PGL}}}
\newcommand{\PSL}{{\operatorname{PSL}}}
\newcommand{\PgammaL}{{\operatorname{P\Gamma L}}}
\newcommand{\SL}{{\operatorname{SL}}}

\newcommand{\SU}{{\operatorname{SU}}}
\newcommand{\PSU}{{\operatorname{PSU}}}

\newcommand{\IF}{\mathbb{F}}

\newcommand{\IZ}{\mathbb{Z}}

\newcommand{\fA}{{\mathfrak{A}}}

\let\lra=\longrightarrow

\let\wt=\widetilde

\newcommand{\lconj}[2]{\,^{#1}\!#2}    

\newtheorem{thm}{Theorem}[section]

\newtheorem{lem}[thm]{Lemma}

\newtheorem{prop}[thm]{Proposition}

\newtheorem{ques}[thm]{Open Question}

\theoremstyle{theorem}

\theoremstyle{definition}

\newtheorem{rem}[thm]{Remark}

\theoremstyle{remark}

\begin{document}


\title{Trivial source endo-trivial modules for finite groups with semi-dihedral Sylow $2$-subgroups}
\date{\today}
\author{{Shigeo Koshitani and Caroline Lassueur}}
\address{{\sc Shigeo Koshitani}, Center for Frontier Science,
Chiba University, 1-33 Yayoi-cho, Inage-ku, Chiba 263-8522, Japan.}
\email{koshitan@math.s.chiba-u.ac.jp}
\address{{\sc Caroline Lassueur}, FB Mathematik, TU Kaiserslautern, Postfach 3049,
         67653 Kaiserslautern, Germany.}
\email{lassueur@mathematik.uni-kl.de}

\thanks{
The first author was partially supported by 
the Japan Society for Promotion of Science (JSPS), Grant-in-Aid for Scientific Research (C) 19K03416, 2019--2021.
The second author acknowledges financial support by DFG SFB/TRR 195. This piece of work is part of Project A18 thereof.
}



\begin{abstract}
We 
finish off the classification of the endo-trivial modules of finite groups with  Sylow $2$-subgroups isomorphic to a semi-dihedral $2$-group started by Carlson, Mazza and Th\'{e}venaz
in their article \textit{Endotrivial modules over groups with quaternion or semi-dihedral Sylow $2$-subgroup} published in 2013.
\end{abstract}

\dedicatory{Dedicated to Jon Carlson on the occasion of his 80th Birthday and to Jacques Th\'{e}venaz on the occasion of his 70th Birthday.}

\maketitle

\pagestyle{myheadings}
\markboth{S. Koshitani and C. Lassueur}{Endo-trivial modules for groups with semi-dihedral Sylow $2$-subgroups}

\section{Introduction}

Endo-trivial modules play an important role in the representation theory of finite groups. For instance in the description of different types of equivalences between block algebras. 
These modules were  introduced by  E. C. Dade  in 1978 \cite{DADE78a,DADE78b}. They have been intensively studied since the beginning of the century and were classified in a number of special cases: e.g. for groups with cyclic, generalised quaternion, Klein-four or dihedral Sylow subgroups, for $p$-soluble groups, for the symmetric and the alternating groups and their Schur covers, for the sporadic groups their Schur covers, or for some infinite  families of finite groups of Lie type.  We refer the reader to the recent survey book \cite{MazzaBook} by N. Mazza  and the references therein for a complete introduction to this theory.\par
An \textit{endo-trivial} module over the group algebra $kG$ of a finite group $G$ over an algebraically closed field~$k$ of prime characteristic~$p$ is by definition a finitely generated $kG$-module whose $k$-endomorphism algebra  is isomorphic to the trivial module in the stable module category. The set of isomorphism classes of indecomposable endo-trivial $kG$-modules form an abelian group under the tensor product $\otimes_k$, which is denoted $T(G)$, and this group is known to be finitely generated. One of the central questions in this theory  is to understand the structure of the group $T(G)$, and, in particular, of its torsion subgroup~$TT(G)$.\par
Now, letting $X(G)$ be the subgroup of $TT(G)$ consisting of all one-dimensional $kG$-modules and  $K(G)$ be the subgroup of $T(G)$ consisting of all the indecomposable endo-trivial $kG$-modules which are at the same time  trivial source $kG$-modules, we have 
$$\Hom(G,k^\times)\cong X(G)\subseteq K(G)\subseteq TT(G)$$
and  $K(G)=TT(G)$ unless a Sylow subgroup is cyclic, generalised quaternion or semi-dihedral (see \cite[Chapter~5]{MazzaBook}). Although it often happens that $X(G)=K(G)$, in general $X(G)\lneq K(G)$. Furthermore, we emphasise that  the determination of the structure of the endo-trivial modules lying in $K(G)\setminus X(G)$ is a very hard problem, to which, up to date, no general solution is known. 
Most of the work that has been done in previous articles provides case by case solutions to the calculation of the abelian group structure of $K(G)$, but in the vast majority of the cases does not provide information about the structure of the modules  in $K(G)\!\setminus\!X(G)$.\par
In \cite{CMT11quat} Carlson-Mazza-Th\'{e}venaz essentially described the structure of the group $T(G)$ of endo-trivial modules for groups with a semi-dihedral Sylow $2$-subgroup. However they left open the question of computing the trivial source endo-trivial modules, i.e. the structure of the subgroup $K(G)$ of $T(G)$. The purpose of the present article is to finish off the determination of $K(G)$ in this case. In order to reach this aim we use three main ingredients, two of which were not available when \cite{CMT11quat} was published:
 \begin{enumerate}
  \item[1.]  The first one is a method we developed in \cite{KoLa15} in order to treat finite groups with dihedral Sylow $2$-subgroups, extended  in \cite{LT17etcentralext} to a more general method to relate the structure of $T(G)$ to that of $T(G/O_{p'}(G))$, which allows us to reduce the problem to groups with $O_{2'}(G)=1$. 
  \item[2.] The second one is the classification of finite groups with semi-dihedral Sylow $2$-subgroups modulo $O_{2'}(G)$ due to Alperin-Brauer-Gorenstein \cite{ABG}. 
  \item[3.] The third main ingredient relies on major new results obtained by J. Grodal through the use of homotopy theory in \cite{GrodalET}, or more precisely on a slight extension of the main theorem of \cite{GrodalET} recently obtained by D. Craven in  \cite{CravenET} and which provides us with purely group-theoretic techniques to deal with Grodal's description of $K(G)$ in \cite[Theorem 4.27]{GrodalET}. The latter results  will in particular enable us to treat families of groups related to the finite groups of Lie type $\SL_3(q)$ with $q\equiv 3\pmod{4}$ and $\SU_3(q)$ with $q\equiv 1\pmod{4}$. 
\end{enumerate}
\noindent With these tools, our main result  is a description of the structure of the group of endo-trivial modules for groups with a semi-dihedral Sylow $2$-subgroup as follows:

\begin{thm}\label{thm:intro}%
Let $k$ be an algebraically closed field of characteristic~$2$ and let $G$ be a finite group with a semi-dihedral Sylow $2$-subgroup of order $2^m$ with $m\geq 4$. Then the following assertions hold. 
\begin{enumerate}
\item[\rm(a)] If $G/O_{2'}(G)\ncong \PGL_2^{\ast}(9)$, then $T(G)\cong X(G)\oplus \IZ/2\IZ\oplus \IZ$.
\item[\rm(b)] If $G/O_{2'}(G)\cong\PGL_2^{\ast}(9)$, then $T(G)\cong K(G)\oplus \IZ/2\IZ\oplus \IZ$
where 
$$K(G)/X(G)\leq \IZ/3\IZ\,.$$ Moreover,  $K(G)/X(G)=1$ if $G=\PGL_2^{\ast}(9)$ and  the bound $K(G)/X(G)\cong \IZ/3\IZ$ is obtained by the group  $G=3.\PGL_2^{\ast}(9)$.
\end{enumerate}
\end{thm}

\noindent To explain the notation, there are exactly three groups $H$ with $\PSL_2(9)< H < \PgammaL_2(9)$ and $|H:\PSL_2(9)|=2$, but amongst them only one has a semi-dihedral Sylow $2$-subgroup and it  is denoted $\PGL_2^\ast(9)$ (see Section~\ref{sec:not}). We also point out that in both cases the  summand $\IZ$ is generated by the first syzygy module of the trivial $kG$-module and the  summand $\IZ/2\IZ$ is generated by a torsion endo-trivial module which is explicitly determined in  \cite[Proposition~6.4]{CMT11quat}. We note for completeness that the latter module had in fact already been constructed via different methods by Okuyama and Kawata, see~\cite[Theorem~5.1]{Kaw93}. \par
The paper is built up as follows. In Section 2 we introduce our notation. In Section 3 we quote main results  on endo-trivial modules which we will use to prove Theorem~\ref{thm:intro}.
In Section 4 we state and prove preliminary results on groups with semi-dihedral Sylow $2$-subgroups. 
In Section 5 we compute the trivial source endo-trivial modules for the special linear and special unitary groups and finally in Section 6 we  prove Theorem~\ref{thm:intro}.


\vspace{2mm}
\section{Notation and definitions}\label{sec:not}

\enlargethispage{5mm}

Throughout this article, unless otherwise specified we adopt the following notation and conventions.  All groups considered are assumed to be finite and all modules over finite group algebras are assumed to be finitely generated left modules. 
We let $k$ denote an algebraically closed field of prime characteristic $p$, we let $G$ denote a finite group of order divisible~by~$p$, and we let $P$ be a Sylow $p$-subgroup of $G$.\\

We denote by $\mathrm{SD}_{2^m}$ the semi-dihedral group of order $2^m$ with $m\geq 4$, by $C_a$ the cyclic group of order $a\geq 1$,  by $\mathfrak A_a$ and $\mathfrak S_a$ the alternating and the symmetric groups on $a$ letters,  
and we refer to the ATLAS \cite{ATLAS} and \cite{ABG} for the definitions of further standard finite groups that occur in the statements of our main results.
For a prime power~$q$ and $\varepsilon\in\{-1,1\}$, we adopt the following notation. 
We set $\PSL^{\varepsilon}_n(q)=\PSL_n(q)$ (resp. $\GL_n^{\varepsilon}(q)=\GL_n(q)$, $\SL_n^{\varepsilon}(q)=\SL_n(q)$) if $\varepsilon=1$ and $\PSL_3^{\varepsilon}(q)=\PSU_n(q)$ (resp. $\GU_n^{\varepsilon}(q)=\GU_n(q)$, $\SU_n^{\varepsilon}(q)=\SU_n(q)$) if $\varepsilon=-1$). 
We note that we consider $\GU_n(q)$ as defined over $\IF_{q^2}$. 
Moreover, we set  $\SL(2,q,\pm 1):=\{A\in \GL_2(q)\mid \det(A)=\pm1\}$ and $\SU(2,q,\pm):=\{A\in \GU_2(q)\mid \det(A)=\pm1\}$. (These groups are denoted $\SL^{\pm}_2(q)$ and $\SU^{\pm}_2(q)$ in \cite{ABG}, but we slightly alter this notation to avoid confusion.) Also, if $q=r^{2f}$ is a positive even power of an odd prime number $r$, then there are exactly three groups $H$ with $\PSL_2(q)< H < \PgammaL_2(q)$ and $|H:\PSL_2(q)|=2$. 
One is $\PGL_2(q)$, one is contained in $\PSL_2(q)\rtimes\langle F\rangle$ where $F$ is the Frobenius automorphism on $\IF_q$, and the third one is denoted $\PGL_2^\ast(q)$ (see  \cite[p.335]{Gor69}).
In this article, we will only work with $\PGL_2^\ast(q)$ and we will essentially need the fact that $\PSL_2(q)$ is normal $\PGL_2^\ast(q)$. Hence we will write $\PGL_2^\ast(q)=\PSL_2(q).2$ without making the difference with the other two extensions. 
For $\emptyset \,{\not=}\,S \subseteq G\ni g,x$ we write 
$^g{\!}S:= \{ gsg^{-1} | s\in S\}$
 and $x^g:= g^{-1}xg$. We let $O_{p'}(G)$, resp. $O_p(G)$, be the largest normal $p'$-subgroup, resp. $p$-subgroup, of $G$ and $O^{p'}(G)$ the smallest normal subgroup of $G$ whose quotient is a $p'$-group. 
Following \cite[Section~2]{CravenET} we define $K_G^{\circ}$ to be the normal subgroup of $N_G(P)$ generated by $N_G(P)\cap O^{p'}(N_G(Q))$ for all $N_G(P)$-conjugacy classes
of subgroups $1<Q\leq P$. Clearly $K_G^{\circ}\unlhd N_G(P)$.\\

We will see the Schur multiplier of $G$ as $H^2(G, \mathbb C^\times)=:M(G)$ and we recall that it is well-known that  $H^2(G, k^\times) \cong  M(G)_{2'}$ (see e.g. \cite[Proposition 2.1.14]{Karp}).
We recall that a \emph{$p'$-representation group} of $G$ (or a \emph{representation group of $G$ relative to $k$}) is a $p'$-central extension $\widetilde{G}$ of $G$ of minimal order with the projective lifting
property.  We emphasise that the kernel of such a $p'$-central extension is isomorphic to $H^2(G, k^\times)$, and  if the Schur multiplier is a $p'$-group, then a $p'$-representation group is just a representation group in the usual sense. For further details on this notion we refer the reader to the expository note \cite{LT17b}.\\

If $M$ is a $kG$-module, then we denote by $M^*$ the $k$-dual of $M$ and by $\End_k(M)$ its $k$-endomorphism algebra, both of which are endowed with the $kG$-module structure given by the conjugation action of $G$.
We recall that a $kG$-module $M$ is called \emph{endo-trivial} if there is an isomorphism of $kG$-modules 
$$\End_k(M)\cong k\oplus \proj\,$$
where $k$ denotes the trivial $kG$-module and $\proj$ denotes a projective $kG$-module or possibly the zero module. Any endo-trivial $kG$-module $M$ decomposes as a direct sum  $M= M_{0}\oplus \proj$ where $M_{0}$, the projective-free part of~$M$, is indecomposable and endo-trivial.
The set $T(G)$ of all isomorphism classes of indecomposable endo-trivial $kG$-modules endowed with the composition law $[M]+[L]:=[(M\otimes_k L)_{0}]$ is an abelian group called the \emph{group of endo-trivial modules of $G$}. 
The zero element is the class $[k]$ of the trivial module and $-[M]=[M^{*}]$. By a result of Puig for finite $p$-groups, extended to arbitrary finite groups by Carlson-Mazza-Nakano, the group $T(G)$ is known to be a finitely generated abelian group (see \cite[Theorem 2.3]{MazzaBook}).\par
We let $X(G)$ denote the group of one-dimensional $kG$-modules endowed with the tensor product $\otimes_k$\,. Clearly $X(G)\leq T(G)$ and we recall that 
$$X(G)\cong \Hom(G,k^{\times})\cong (G/[G,G])_{p'}\,.$$
In particular, it follows that $X(G)=\{[k]\}$ when $G$ is $p'$-perfect. Furthermore, we let $K(G)$ denote the subgroup of $T(G)$ consisting of the isomorphism classes of indecomposable endo-trivial $kG$-modules with a  trivial source.  It follows easily from the theory of vertices and sources that $K(G)$ is precisely the kernel of the restriction homomorphism
$$
\Res_P^G : T(G) \longrightarrow T(P), [M]\mapsto[\Res^G_P(M)_0] \,.
$$
Clearly  $X(G)\leq K(G)$ since the restriction of a one-dimensional $kG$-module to $P$ is the trivial module. Moreover, $K(G)$ is a finite group as  it is made up of isomorphism classes of trivial source $kG$-modules, hence $K(G)\leq TT(G)$ . Moreover, by the main result of~\cite{CT}, we have $K(G)=TT(G)$~unless $P$ is cyclic, generalised quaternion, or semi-dihedral.\\


\vspace{2mm}
\section{Known results}

To begin with, we quickly review the results about $T(G)$ in the semi-dihedral case obtained by Carlson-Th\'evenaz in \cite{CT} and Carlson-Mazza-Th\'{e}venaz in~\cite{CMT11quat}.

\begin{thm}\label{thm:CMT}
Let $p=2$ and let $G$ be a finite group with a Sylow $2$-subgroup $P\cong \mathrm{SD}_{2^m}$ of order $2^m$ with $m\geq 4$. Then the following assertions hold.
\begin{enumerate}
\item[\rm(a)] \cite[Theorem~7.1]{CT} $T(P)\cong \IZ/2\IZ\oplus \IZ$.
\item[\rm(b)] \cite[Proposition~6.1]{CMT11quat}  $T(G)\cong K(G)\oplus \Image(\Res^G_P)$.
\item[\rm(c)] \cite[Proposition~6.4]{CMT11quat}  $\Res^G_P:T(G)\lra T(P)$ is a split surjective group homomorphism.
\item[\rm(d)] 
 \cite[Proposition~6.4]{CMT11quat} $TT(G)\cong K(G)\oplus\IZ/2\IZ$, where the $\IZ/2\IZ$ summand is generated by a self-dual torsion endo-trivial module which is not a trivial source module. 
\item[\rm(e)] \cite[Corollary~6.5]{CMT11quat}  If $P=N_G(P)$, then $K(G)=\ker(\Res^{G}_{P})=
\{[k]\}$
and hence $T(G)\cong T(P)\cong \IZ/2\IZ\oplus \IZ$. 
\end{enumerate}
\end{thm}

\begin{ques}
The remaining open question in the article~\cite{CMT11quat} by Carlson-Mazza-Th\'{e}venaz about finite groups $G$ with semi-dihedral Sylow $2$-subgroups  is to compute the structure of the group $K(G)$ when the Sylow $2$-subgroups are not self-normalising.
 
\end{ques}

\noindent  Next, we state below the main results which we will use in our proof of Theorem~\ref{thm:intro}. However, they are all valid in arbitrary prime characteristic~$p$. 

\begin{lem}[{}{\cite[Lemma 2.6]{MT07}}]\label{lem:Op(G)}
Let $G$ be a finite group and  let $P$ be a Sylow $p$-subgroup of $G$. 
If $\lconj{x}{P}\cap P\neq 1$ for every $x\in G$, then $K(G)=X(G)$. In particular, if  $O_p(G)>1$, then $K(G)=X(G)$. 
\end{lem}

\noindent Then, the following lemma will  be applied to semi-direct products of the form  $G=N\rtimes H$ with $p\nmid |G:N|$. 

\begin{lem}[{}{\cite[Theorem 5.1.(4.)]{MazzaBook}}]\label{lem:normal}
Let $G$ be a finite group and let $N\unlhd G$ such that $p\nmid |G:N|$. If  $K(N)=X(N)$, then $K(G)=X(G)$.
\end{lem}

\noindent  We will also use the following result from our previous paper on endo-trivial modules for groups with dihedral Sylow $2$-subgroups.

\begin{thm}[{}{\cite[Theorem 1.1]{KoLa15}}]\label{thm:KoLa15}
Let $G$ be a finite group with $p$-rank at least $2$ and which does not admit a  strongly $p$-embedded subgroup.
Let $H\unlhd G$ be a normal subgroup such that  $p\nmid |H|$. If $H^2(G, k^{\times}) = 1$, then
$$K(G) = X(G) + \Inf_{G/H}^{G}(K(G/H))\,.$$
\end{thm}

\noindent In case $H^2(G, k^{\times}) \neq 1$, then we may apply the following generalisation of the above result.

\begin{thm}[{}{\cite[Theorem 1.1]{LT17etcentralext}}]\label{thm:LT17}
Let $G$ be a finite group with $p$-rank at least $2$ and which does not admit a  strongly $p$-embedded subgroup.
Let $\wt{Q}$ be any $p'$-representation group of the group $Q:=G/O_{p'}(G)$.
\begin{enumerate}
\item[\rm(a)] There exists an injective group homomorphism   
$$\Phi_{G,\wt{Q}}: T(G)/X(G)\lra T(\wt{Q})/X(\wt{Q})\,.$$
In particular, $\Phi_{G,\wt{Q}}$ maps the class of $\Inf_Q^G(W)$ to the class of~$\Inf_Q^{\wt{Q}}(W)$, for any endo-trivial $kQ$-module~$W$.
\item[\rm(b)] The map $\Phi_{G,\wt{Q}}$ induces by restriction an injective group homomorphism   
$$\Phi_{G,\wt{Q}}: K(G)/X(G)\lra K(\wt{Q})/X(\wt{Q})\,.$$
\item[\rm(c)] In particular, if $K(\wt{Q})= X(\wt{Q})$, then $K(G)= X(G)$.\\
\end{enumerate}
\end{thm}


\noindent Finally, we will apply a recent result obtained by David Craven in \cite{CravenET}, which gives a purely group-theoretic method in order to use  Grodal's homotopy-theoretical description of $K(G)$ in   \cite[Theorem 4.27]{GrodalET}.

\begin{lem}[{}{\cite[Section~2]{CravenET}}]\label{lem:Craven}
Let $G$ be a finite group and let $P$ be a Sylow $p$-subgroup of $G$. If $K_G^{\circ}=N_G(P)$, then $K(G)\cong \{[k]\}$.
\end{lem}

\begin{proof}
Assuming $K_G^{\circ}=N_G(P)$, by \cite[Theorem~2.3 and the remark before Theorem~2.3]{CravenET} we have 
$$K(G)\cong \left(N_G(P)/K_G^{\circ}\right)^{\text{ab}}=1\,.$$ The claim follows. 
\end{proof}


\vspace{4mm}
\section{Some properties of groups with semi-dihedral Sylow $2$-subgroups}

From this point forward and for the remainder of this article, we assume that the field~$k$ has characteristic $p=2$.\\

\noindent 
The following result provides a classification of finite groups with semi-dihedral Sylow $2$-subgroups, which was obtained by Benjamin Sambale in an upublished note as a byproduct of the results of Alperin-Brauer-Gorenstein in \cite{ABG}. A proof can be found in \cite[Theorem~3.1]{KLS20}.

\begin{prop}[{}{\cite[Theorem~3.1]{KLS20}}]\label{prop:classificSD}
Let $G$ be a finite group with a  Sylow $2$-subgroup isomorphic to ${\mathrm{SD}}_{2^m}$ ($m\geq 4$) and $O_{2'}(G) = 1$. Let $r$ be a prime number and $q=r^n$ be a positive power of $r$. Then one of the following holds:
\begin{enumerate}
\item[{\bf \rm(SD1)}] $G\cong {\mathrm{SD}}_{2^m}$;
\item[{\bf \rm(SD2)}] $G\cong\mathrm{M}_{11}$ and $m=4$;
\item[{\bf \rm(SD3)}] $G\cong \SL(2,q,\pm 1)\rtimes C_d$ where  $q\equiv -1\pmod{4}$, $d\mid n$ is odd and $(q + 1)_2 = 2^{m-2}$;
\item[{\bf \rm(SD4)}] $G\cong \SU(2,q,\pm 1)\rtimes C_d$ where $q\equiv 1\pmod{4}$, $d\mid n$ is odd and and $(q - 1)_2 = 2^{m-2}$;
\item[{\bf \rm(SD5)}] $G\cong\PGL^{\ast}_2(q^2)\rtimes C_d$ where $r$ is odd, $d\mid n$ is odd and  $(q^2 - 1)_2 = 2^{m-1}$;
\item[{\bf \rm(SD6)}] $G\cong\PSL^{\varepsilon}_3(q).H$ where $q\equiv -\varepsilon\pmod{4}$, $H\leq C_{(3,q-\varepsilon)}\times C_n$ has odd order and $(q + \varepsilon)_2 = 2^{m-1}$.
\end{enumerate}
\end{prop}

\noindent In addition, two crucial results for the present article are given by the following lemma and proposition, which will allow us to apply Theorem~\ref{thm:KoLa15} and Theorem~\ref{thm:LT17}.

\begin{lem}\label{lem:stpemb}
A finite group with a semi-dihedral Sylow $2$-subgroup does not admit any strongly $2$-embedded subgroup.
\end{lem}

\begin{proof}
The Bender-Suzuki theorem \cite[Satz~1]{Bender71} 
states that a finite group  $G$ with a strongly $2$-embedded subgroup $H$ has one of the following forms:
\begin{itemize}
\item[\rm1.] $G$ has cyclic or generalised quaternion Sylow 2-subgroups and $H$ contains the centraliser of an involution; or 
\item[\rm2.] $G/O_{2'}(G)$ has a normal subgroup of odd index isomorphic to one of the simple groups $\PSL_2(q)$, $\mathrm{Sz}(q)$ or $\PSU_3(q)$ where $q\geq 4$ is a power of $2$ and $H$ is $O_{2'}(G)N_G(P)$ for a Sylow $2$-subgroup $P$ of $G$.
\end{itemize}
Therefore, it follows from Proposition~\ref{prop:classificSD} that such a group cannot admit a semi-dihedral Sylow $2$-subgroup. 
\end{proof}

\begin{prop}\label{prop:Smult}
Let $G$ be a finite group with a semi-dihedral Sylow $2$-subgroup $P\cong {\mathrm{SD}}_{2^m}$  for some $m\geq 4$ and $O_{2'}(G)=1$. Let $r$ be a prime number and $q=r^n$ be a positive power of $r$.
 Then the following assertions hold.
\begin{enumerate}
\item[\rm(a)] If $G={\mathrm{SD}}_{2^m}$, then \smallskip $H^2(G, k^\times)=1$.

\item[\rm(b)] If $G={\mathrm M}_{11}$, then  \smallskip  $H^2(G, k^\times)=1$.

\item[\rm(c)] If $G=\SL(2,q,\pm 1)\rtimes C_d$ where  $q\equiv -1\pmod{4}$ and $d\mid n$ is odd, then  \smallskip ${H^2(G, k^\times)=1}$.

\item[\rm(d)] If $G=\SU(2,q,\pm 1)\rtimes C_d$ where   $q\equiv 1\!\pmod{4}$ and $d\mid n$ is odd, then  \smallskip  ${H^2(G, k^\times)\!=\!1}$, unless $G=\SU(2,9,\pm 1)$, in which case  \smallskip  $H^2(G, k^\times)\cong C_3$.

\item[\rm(e)]  If $G=\PGL_2^*(q^{2})\rtimes C_d$ where $q^{2}$ is odd and $d\mid n$ is odd,  then  \smallskip  ${H^2(G, k^\times)=1}$, unless $G=\PGL_2^*(9)$, in which case \smallskip $H^2(G, k^\times)\cong C_3$.

\item[\rm(f)]
If $G=\PSL^{\varepsilon}_3(q).H$ where  $q\equiv -\varepsilon \pmod{4}$ and  $H\leq C_{(3,q-\varepsilon)}\times C_n$ has odd order, then
 $|H^2(G,k^{\times})| \big| (3,q-\varepsilon)$ if $H$ is cyclic, and $|H^2(G,k^{\times})| \big| 9$ else.
\end{enumerate}
\end{prop}

\begin{proof}
Set  $h:=|H^2(G, k^\times)|$. In order to compute $h$ we recall that $H^2(G, k^\times)\cong M(G)_{2'}$ (see \cite[Proposition 2.1.14]{Karp}) and  if $N\unlhd G$ such that $G/N$ is cyclic then, by \cite[Theorem~3.1(i)]{JONES}, we have
\[
|M(G)|\, \Big|\, |M(N)|\cdot |N/[N,N]| \,. \tag{$\ast$}
\]
We compute:
\begin{enumerate}
\item[\rm(a)] Because the Schur multiplier of a cyclic group is trivial, we obtain from 
\cite[Theorem 2.1.2(i)]{Karp} (or \cite[Corollary 5.4]{ISAACS}) that if $p$ is an odd prime divisor of $|M(G)|$ then a Sylow $p$-subgroup of $G$ must be noncyclic, hence \smallskip $|M(G)_{2'}|=1$.
\item[\rm(b)] See the \smallskip  ATLAS \cite[p.18]{ATLAS}. 
\item[\rm(c)]  First, we have  $\SL(2,q,\pm 1)\cong\SL_2(q).2$ (see \cite[Chapter I, p.4]{ABG}). Thus ($\ast$) yields 
$$M( \SL(2,q,\pm 1) )=1$$
since  $\mathrm{SL}_2(q)$ is perfect and has a trivial Schur multiplier as  $q\equiv 3 \pmod{4}$ (see \cite[7.1.1.Theorem]{Karp}).
Therefore, we may apply ($\ast$) again to $G=N\rtimes C_d$ with $N= \SL(2,q,\pm 1)$. Because $|N/[N,N]|=2$ we obtain
$$|M(G)| \,\Big|\, |M(N)|{\cdot} |N/[N,N]|=2$$
and it follows that \smallskip  $h=|M(G)_{2'}|=1$.
\item[\rm(d)] We have ${\mathrm{SU}}(2,q,\pm1) \cong {\mathrm{SU}}_2(q).2$ (see \cite[p.4]{ABG}) and 
${\mathrm{SU}}_2(q)\cong {\mathrm{SL}}_2(q)$. Therefore, if $q\neq 9$ we have  $M(\SU_2(q))\cong M(\SL_2(q))=1$ (see \cite[7.1.1.Theorem]{Karp}) and  the claim follows by the same argument as  in~{\rm(c)}, applying ($\ast$)  twice.\\
Next, if $q=9$, then we first observe that necessarily $d=1$. Moreover, we claim  that $\SU(2,9,\pm 1)\cong 2.\PGL_2(9)$. Indeed, as $\PSL_2(9)\cong \fA_6$,  by \cite[diagram on~p.4]{ABG}, 
$Q:=\SU(2,9,\pm1) \cong C_2.\PSL_2(9).C_2\cong C_2.\mathfrak A_6.C_2$\,. Hence, we read from the ATLAS  \cite[p.4]{ATLAS} that 
$$Q\in\{2.\mathfrak A_6.2_1 = 2.\mathfrak S_6, \ 
2.\mathfrak A_6.2_2=2.\PGL_2(9),\,2.\mathfrak A_6.2_3=2.\PGL_2^{\ast}(9) \}\,.$$ 
(Here we use the ATLAS notation.) 
Considering the quotient  $\bar Q:=Q/Z(Q)=Q/C_2$ and letting $\bar P\in\Syl_2(\bar Q)$, we have
$\bar  P \cong D_8\!\times\!C_2$ if $Q\cong 2.\mathfrak S_6$, $\bar P\cong \mathrm{D}_{16}$ if $Q\cong 2.\PGL_2(9)$, and $\bar P\cong \mathrm{SD}_{16}$ if ${Q\cong 2.\PGL_2^{\ast}(9)}$\,. 
However, as a Sylow $2$-subgroup of $Q$ is semi-dihedral of order~32 and  $\mathrm{SD}_{32}/Z(\mathrm{SD}_{32}) \cong \mathrm{D}_{16}$, we must have $Q\cong 2.\PGL_2(9)$.  
Hence, by \cite[2.1.15 Corollary]{Karp} and the ATLAS \cite[p.4]{ATLAS}, we obtain that 
$$M(G)_{2'}\cong  C_3\,.$$ 
\item[\rm(e)] We have ${\mathrm{PGL}}^*_2(q^{2})={\mathrm{PSL}}_2(q^{2}).2$ (see  \cite[p.335]{Gor69}). 
Now, if $q^{2}\neq 9$, then  $\PSL_2(q^{2})$ is perfect and has a Schur multiplier of order $2$. Therefore applying ($\ast$) twice as in (c) it follows that $h=|M(G)_{2'}|=1$.
If $q^2=9$, then necessarily $d=1$, and  because ${\mathrm{PGL}}^*_2 (9) \cong \fA_6.2$ we read from the ATLAS \cite[p.4]{ATLAS} that   \smallskip  $h=3$.
\item[\rm(f)] Write $G=N.H$ with $N:=\PSL^{\varepsilon}_3(q)$ ($q\equiv -\varepsilon \pmod{4}$). 
First assume that $H$ is cyclic. Because $N$ is perfect, by  ($\ast$), we have that
$$|M(G)|\, \Big|\, |M(N)|=(3,q-\varepsilon)\,$$
as $M(N)\cong C_{(3,q-\varepsilon)}$  (see e.g. \cite[Chapter~2]{ATLAS}). Hence $h=|M(G)_{2'}| \big| (3,q-\varepsilon)$\,. 
Next, if  $H$ is not cyclic, we have  $H\cong C_3\times C_a$ with $3\mid a$.
For  $X:=N.C_3$  we obtain that  $|M(X)|\big| |M(N)|=3$ by  ($\ast$).
Then  applying ($\ast$) a second time, we get
$$|M(G)|\, \Big|\, |M(X)|\cdot |X/[X,X]|=9\,.$$
Hence $h=|M(G)_{2'}| \big| 9$\,.
\end{enumerate}
\end{proof}

\begin{rem}
We note that if in case (SD6) of Proposition~\ref{prop:classificSD} the extension $G=\PSL_3^{\varepsilon}(q).H$ is split, then it follows from a general result of K. Tahara \cite[Theorem 2]{Tahara} on the second cohomology groups of semi-direct products that $H^2(G,k^{\times})\cong C_{(3,q-\varepsilon)}$ if $H$ is cyclic and $H^2(G,k^{\times})\cong C_3\times C_3$ if $H$ is non-cyclic.
\end{rem}


\enlargethispage{5mm}

\vspace{2mm}
\section{Endotrivial modules for $\SL^{\varepsilon}_3(q)$ with $q\equiv -\varepsilon\pmod{4}$}

In this section, we compute $K(G)$ for $G=\SL^{\varepsilon}_3(q)$ with $q\equiv -\varepsilon\pmod{4}$ in characteristic~$2$. 
We note that for the special linear group $G=\SL_3(q)$ with $q\equiv -1\pmod{4}$ the structure of $K(G)$ is  given in \cite[Theorem 9.2(b)(ii)]{CMN16}, namely $K(G)=X(G)$.  However, the proof given to this fact in \cite{CMN16} contains an error:  the authors assumed that a Sylow $2$-subgroup $S$ of the simple group $H=\PSL_3(q)$ is always self-normalising, which is not  correct in general, as  $N_H(S)\cong S\times Z$ where $Z$ is a cyclic group of order $(q-1)_{2'}/(q-1,3)$\,. (See e.g. \cite[Corollary on p.2]{Kond05}). More precisely, identifying $\GL_2(q)$ with the subgroup of $\SL_3(q)$ made up of the $2\times 2$ left upper block matrices as below, we have $Z=O_{2'}\big(Z(\GL_2(q))\big)/U(\SL_3(q))$.  For this reason, we treat both $\SL_3(q)$ and $\SU_3(q)$.\\

Throughout  this section we let $G:=\SL^{\varepsilon}_3(q)$ and  $\widetilde{G}:=\GL_3^{\varepsilon}(q)$  with $q\equiv -\varepsilon\pmod{4}$ an odd prime power and we
define $\overline{q}$ to be $q$ if $\varepsilon= 1$ and $q^2$ if $\varepsilon=-1$.  Furthermore, we let
$$\imath:\GL_2^{\varepsilon}(q)\lra \SL_3^{\varepsilon}(q),\, a\mapsto \left(\begin{matrix}a & 0 \\ 0 & \det(a)^{-1}\end{matrix}\right)$$
be the natural embedding. 
Now, in order to describe the normaliser of $P\in\Syl_2(G)$, we follow the procedure described in \cite[Sections~7~and~8]{ST18}  to obtain $N_G(P)$ from the normaliser $N_{\wt{G}}(\wt{P})$ of a Sylow $2$-subgroup $\wt{P}$  of $\wt{G}$ as given by Carter-Fong \cite{CarterFong}. 
Firstly, as the $2$-adic expansion of~$3$ is $3=2^{r_1}+2^{r_2}$ with $r_1=1$ and $r_2=0$ we have 
 $\widetilde{P}\cong \prod_{i=1}^2 S^{\varepsilon}_{r_i}(q)$ where $S^{\varepsilon}_{r_i}(q)\in \Syl_2(\GL^{\varepsilon}_{2^{r_i}}(q))$ and 
$$N_{\wt{G}}(\wt{P})\cong \wt{P}\times C_{(q-\varepsilon)_{2'}}\times C_{(q-\varepsilon)_{2'}}\,.$$
Concretely, we may assume that $\wt{P}$ is realised by embedding  $\prod_{i=1}^2 S^{\varepsilon}_{r_i}(q)\leq \prod_{i=1}^2 \GL^{\varepsilon}_{2^{r_i}}(q)$ block-diagonally in a natural way. Moreover, for $1\leq j\leq 2$ the corresponding factor $C_{(q-\varepsilon)_{2'}}$ is  embedded as $O_{2'}(Z( \GL^{\varepsilon}_{2^{r_j}}(q) ))$, so that an arbitrary element of $N_{\wt{G}}(\wt{P})$ is of the form $xz$ with $x\in\wt{P}$ and $z$ is a diagonal matrix of the form
$z=\diag(\lambda_1 I_2, \lambda_2)$  with $\lambda_1,\lambda_2\in C_{(q-\varepsilon)_{2'}}\leq \IF^{\times}_{\overline{q}}$ and $I_2$ the identity matrix in $\GL_2(\overline{q})$.

\begin{lem}\label{lem:NormaliserSL3eps}
The following assertions hold.
\begin{enumerate}
\item[\rm(a)] $P:=\wt{P}\cap G=\{\imath(x)\mid x\in S_{1}^{\varepsilon}(q)\}$ is a Sylow $2$-subgroup of $G$ which is normal \smallskip in~$\wt{P}$\,.
\item[\rm(b)]  $N_G(P) = P\times O_{2'}(C_G(P))$ where 
$$O_{2'}(C_G(P))=\{\diag(\lambda_1 I_2,\lambda_1^{-2})\mid \lambda_1\in C_{(q-\varepsilon)_{2'}}\leq\IF_{\overline{q}}^{\times}\}\cong C_{(q-\varepsilon)_{2'}}\,.$$
\item[\rm(c)]  If  $Q\leq G$ is such that $Q\cong C_2\times C_2$, 
then  $O^{2'}(N_G(Q))=N_G(Q)$\,.
\end{enumerate}
\end{lem}

\begin{proof}
Part (a) is given by \cite[Sections~\S8.1]{ST18}. For part (b), first, as $P$ is semi-dihedral its automorphism group $\Aut(P)$ is a $2$-group and it follows that  
$$N_G(P)=PC_G(P)=P\times Z\qquad \text{with }Z:=O_{2'}(C_G(P))\,.$$ 
Now by \cite[Sections~\S8.1]{ST18}, we have $N_G(P) =  N_{\widetilde{G}}(\widetilde{P})\cap G$ and the claim follows from the above description of $N_{\widetilde{G}}(\widetilde{P})$. \\
Part (c) is obtained as follows. Let $\zeta$ is a generator of the subgroup $C_{(q-\varepsilon)_{2'}}\leq \IF_{\overline{q}}^{\times}$.
Consider the diagonal matrices
\[
u:=\diag(1, -1, -1), v:=\diag(-1, 1, -1),
x:=\diag(\zeta,\zeta,\zeta^{-2}), y:=\diag(\zeta,1,\zeta^{-1})
\in G.
\]
We can set $Q:=\langle u,v\rangle\leq G$ since all subgroups of $G$ 
isomorphic to $C_2\times C_2$ are $G$-conjugate. 
Furthermore, consider the matrices 
$$ t:= \begin{pmatrix} 1&0&0\\ 0&0&1\\ 0&-1&0\end{pmatrix},\quad \,a:= \begin{pmatrix}0&1&0\\0&0&1\\1&0&0\end{pmatrix} \in G\,$$
and set $H:=O^{2'}(N_G(Q))$.
Then we have
$$
a^t = a^2 u v,  \ u^a=v, \ v^a= uv, \ (uv)^a=u, \
x^a = xy^{-3}, \ x^t= x^{-2}y^3, \ y^a=xy^{-2}, \ y^t=x^{-1}y^2.
$$
Clearly $u, v, t\in H$ since $H$ is generated by the $2$-elements of $N_G(Q)$.
Moreover $a\in H$, so that $x, y\in H$ as well. The assertion is proved.
\end{proof}

\begin{prop}\label{prop:SU3}
If $G=\SL^{\varepsilon}_3(q)$ with $q\equiv -\varepsilon\pmod{4}$ an odd prime power, then 
$$K(G)=X(G)=\{[k]\}\,.$$
\end{prop}

\begin{proof}
Let $P$ and $Q$ be as in Lemma~\ref{lem:NormaliserSL3eps}. 
Thanks to Lemma~\ref{lem:Craven}(b) it is enough to prove that  $N_G(P)=K_G^{\circ}$.
By Lemma~\ref{lem:NormaliserSL3eps}(b),
we have $N_G(P) = P\times O_{2'}(C_G(P))$, hence it is clear that 
$P=O^{2'}(N_G(P))=N_G(P)\cap O^{2'}(N_G(P)) \leq  K_G^{\circ}\,.$ 
Now, 
$$O_{2'}(C_G(P))\leq C_G(P)\leq C_G(Q)\leq N_G(Q)=O^{2'}(N_G(Q))$$
where the last equality holds by Lemma~\ref{lem:NormaliserSL3eps}(c). Thus, 
$$O_{2'}(C_G(P)) \leq N_G(P)\cap O^{2'}(N_G(Q))\leq K_G^{\circ}\,.$$
\end{proof}


\section{Proof of Theorem~\ref{thm:intro}}

Throughout this section we let $G$ denote a finite group such that  $P\in\Syl_2(G)$ is semi-dihedral of order $2^{m}$ for some $m\geq 4$.
We let $r$ be a prime number and $q=r^n$ be  a positive power of $r$.  We write $Q:=G/O_{2'}(G)$.\\

\noindent In order to prove Theorem~\ref{thm:intro}, we go through the possibilities for $Q$ given by Proposition~\ref{prop:classificSD}. As the $2$-rank of $P$ is $2$ and $G$ does not have any strongly $2$-embedded subgroup by Lemma~\ref{lem:stpemb}, we may apply Theorem~\ref{thm:KoLa15} and Theorem~\ref{thm:LT17}. \\

\begin{prop}\label{prop:proofA(a)}
If $G/O_{2'}(G)\ncong \PGL_2^{\ast}(9)$, then \smallskip $K(G)=X(G)$.
\end{prop}

\begin{proof}
To begin with, we assume that $H^2(Q,k^{\times})=1$. In this case it suffices to prove that $K(Q)=X(Q)$, since it then follows from Theorem~\ref{thm:KoLa15}  that $K(G)=X(G)$. 
We go through the possibilities for $Q$ given by Proposition~\ref{prop:classificSD} and Proposition~\ref{prop:Smult} as \smallskip follows. \\
\textsf{Case 1}: $Q\cong P$. Then clearly $K(Q)=\{[k]\}$. (Although in this case $G$ is $2$-nilpotent and hence the fact that $K(G)=X(G)$ is proved \smallskip in~\cite[Theorem 5.1]{CMTpsol}).\\
\textsf{Case 2}: $Q\cong  \mathrm{M}_{11}$. As $\mathrm{M}_{11}$ has a self-normalising Sylow $2$-subgroup it follows from Theorem~\ref{thm:CMT}(e) that $K(Q)=X(Q)=\{[k]\}$.  \smallskip (See  \cite[\S 4.1]{LaMaz15}.)\\
\textsf{Case 3}: $Q\cong  \SL(2,q,\pm 1)\rtimes C_d$ with $q\equiv -1\pmod{4}$ and $d\mid n$ is odd. In this case set $N:= \SL(2,q,\pm 1)$. Then $K(N)=X(N)$ by Lemma~\ref{lem:Op(G)} because $G$ admits a non-trivial central $2$-subgroup $\{\pm I_2\}$. 
Therefore, Lemma~\ref{lem:normal} yields  $K(Q)=X(Q)$, because $N\unlhd Q$ with \smallskip odd~index. \\
\textsf{Case 4}: $Q\cong \SU(2,q,\pm 1)\rtimes C_d$ with $9\neq q\equiv 1\pmod{4}$ and $d\mid n$ is odd.   In this case set $N:= \SU(2,q,\pm 1)$. Then $K(N)=X(N)$ by Lemma~\ref{lem:Op(G)} because $G$ admits a non-trivial central $2$-subgroup $\{\pm I_2\}$. Therefore, Lemma~\ref{lem:normal} yields  $K(Q)=X(Q)$, because $N\unlhd Q$ with \smallskip odd~index. \\
\textsf{Case 5}: $Q\cong \PGL_2^{\ast}(q^{2})\rtimes C_d$ where $q^{2}\,{\not=}\,9$ is odd and $d$ is an odd divisor of $n$.  In this case set $N:=\PGL_2^{\ast}(q^{2})\cong \PSL_2(q^{2}).2$\,. As a Sylow $2$-subgroup of $\PSL_2(q^{2})$  is self-normalising, so is a  Sylow $2$-subgroup of $N$. Therefore  $K(N)=X(N)$ by Theorem~\ref{thm:CMT}(e) and hence it follows from  Lemma~\ref{lem:normal} that \smallskip $K(Q)=X(Q)$. \\
\textsf{Case 6}: $Q\cong \PSL^{\varepsilon}_3(q).H$ where   $q\equiv -\varepsilon\pmod{4}$ and  $H\leq C_{(3,q-\varepsilon)}\times C_n$ has odd order and $H^2(Q,k^{\times})=1$. In this case set $N:= \PSL^{\varepsilon}_3(q)$. 
As $K(\SL^{\varepsilon}_3(q))=X(\SL^{\varepsilon}_3(q))$ by Proposition~\ref{prop:SU3} and $\SL^{\varepsilon}_3(q)$ is a $2'$-representation group of $\PSL^{\varepsilon}_3(q)$   we also have that $K(N)=X(N)$ by Theorem~\ref{thm:LT17}(c). Therefore $K(Q)=X(Q)$ by Lemma~\ref{lem:normal} since  $N\unlhd Q$ with odd \smallskip index. \\
Next we assume that $H^2(G,k^{\times})\neq 1$ and let $\wt{Q}$ be a $2'$-representation group of $Q$.  In this case it suffices to prove that $K(\widetilde{Q})=X(\widetilde{Q})$, 
because  it then follows from Theorem~\ref{thm:LT17}(c) that  that $K(G)=X(G)$.  We go through the possibilities for $Q$ given by Proposition~\ref{prop:classificSD} and Proposition~\ref{prop:Smult} as \smallskip follows.\\
\textsf{Case 7}: $Q\cong  \SU(2,9,\pm 1)$.  
First, recall that  $\SU(2,9,\pm 1)$ is isomorphic to $2.\PGL_2(9)$\,. 
(See proof of Proposition \ref{prop:Smult}(d)).
By Proposition~\ref{prop:Smult}(d) we have  $H^2(Q,k^{\times})\cong C_3$ and  we may choose $\wt{Q}$ to be $6.\PGL_2(9)$. Therefore Lemma~\ref{lem:Op(G)} yields $K(\wt{Q})=X(\wt{Q})$ because $\wt{Q}$ admits a normal (central) subgroup of \smallskip order~$2$. \\
\textsf{Case 8}: $Q\cong  \PSL^{\varepsilon}_3(q).H$ where   $q\equiv -\varepsilon\pmod{4}$ and  $H\leq C_{(3,q-\varepsilon)}\times C_n$ has odd order and $|H^2(Q,k^{\times})|=3$. Then we can take $\widetilde{Q}=N.H$ with $N:=\SL^{\varepsilon}_3(q)$. Now, $K(N)=X(N)$ by Proposition~\ref{prop:SU3}, so that $K(\widetilde{Q})=X(\widetilde{Q})$ by Lemma~\ref{lem:normal} because $N$ is normal of odd index \smallskip in~$\widetilde{Q}$.\\
\textsf{Case 9}: $Q\cong  \PSL^{\varepsilon}_3(q).H$ where   $q\equiv -\varepsilon\pmod{4}$ and  $H\leq C_{(3,q-\varepsilon)}\times C_n$ has odd order and $|H^2(Q,k^{\times})|=9$. 
 In this case $\wt{Q}$  is a central extension of $Q_1:=\SL^{\varepsilon}_3(q).H$ with kernel $Z\cong C_3$. Since $\SL^{\varepsilon}_3(q)$ is normal in $Q_1$, there exists a normal subgroup $Y\unlhd \wt{Q}$ containing $Z$ and such that $Y/Z\cong \SL^{\varepsilon}_3(q)$ and as $\SL^{\varepsilon}_3(q)$ is its own $2'$-representation group, we have that $Y\cong Z\times \SL^{\varepsilon}_3(q)$.
Therefore, as $K(\SL^{\varepsilon}_3(q))=X(\SL^{\varepsilon}_3(q))$ by Proposition~\ref{prop:SU3},  it follows from Lemma~\ref{lem:normal}, applied a first time, that $K(Y)=X(Y)$ and applied a second time that $K(\wt{Q})=X(\wt{Q})$. 
\end{proof}

\newpage

\begin{lem}\label{lem:SD5b}{\ }
\begin{enumerate}
\item[\rm(a)] If $G=\PGL_2^{\ast}(9)$, then \smallskip $K(G)=X(G)$.
\item[\rm(b)] If $G=3.\PGL_2^{\ast}(9)$, then $K(G)\cong \IZ/3\IZ$. Moreover, the indecomposable representatives of the two non-trivial elements of $K(G)$ lie in the two distinct faithful and dual $2$-blocks with full defect $B$ and $B^{\ast}$ of $G$. Their Loewy and socle series  are respectively
$$ \boxed{\footnotesize \begin{matrix} 9 \\ 9 \ \ 6 \\ 9\end{matrix}} \quad\text{ and }\quad
\boxed{\footnotesize \begin{matrix} 9^* \\ 9^* \ \ 6^* \\ 9^*\end{matrix}} 
$$
where $9$ (resp. $6$) denotes the unique $9$-dimensional (resp. $6$-dimensional) simple \smallskip $kB$-module and $6^\ast$ (resp. $9^\ast$)  is the $k$-dual of $6$ (resp. $9$). 
\item[\rm(c)]If $G/O_{2'}(G)\cong \PGL_2^{\ast}(9)$, then $K(G)/X(G)$ is isomorphic to a subgroup of  $\IZ/3\IZ$. 
\end{enumerate}
\end{lem}

\begin{proof}
\begin{enumerate}
\item[\rm(a)] The group $\PGL_2^{\ast}(9)$ has a semi-dihedral Sylow $2$-subgroup $P\cong \text{SD}_{16}$, which is self-normalising. Hence the claim is immediate from \smallskip Theorem~\ref{thm:CMT}(e). 
\item[\rm(b)] 
We treat the group $G=3.\PGL_2^{\ast}(9)$ entirely via computer algebra using MAGMA \cite{MAGMA}.\\
First, we note that the block structure of $G$ can be found in the \emph{decomposition matrices} section at the Modular Atlas Homepage \cite[$A_6.2_3$]{ModAtlas}, where we also read that $B$ has exactly two simple modules, one of dimension~$6$ and one of dimension~$9$.\\
Next,  in this case, $N_G(P)=P\times Z(G)\cong P\times C_3$, hence by the definition of $K(G)$ we need to show that the Green correspondents $X_1$ and $X_2$ of the two non-trivial $1$-dimensional $kN_G(P)$-modules are endo-trivial $kG$-modules. However, as $X_1$ and $X_2$ must be dual to each other, it suffices to consider~$X_1$. Computing $X_1$ with MAGMA we find that $X_1$ has the given Loewy (and socle) series and that its  restriction to $P$ is an endo-trivial module and therefore so \smallskip is $X_1$.
\item[\rm(c)] For $Q\cong\PGL_2^{\ast}(9)$ we take $\widetilde{Q}=3.\PGL_2^{\ast}(9)$ and the claim follows from Theorem~\ref{thm:LT17}(b) together with {\rm(b)}.
\end{enumerate}
\end{proof}

\noindent We can now prove our main Theorem.

\begin{proof}[Proof of Theorem~\ref{thm:intro}]
We know from Theorem~\ref{thm:CMT}(a)--(d) that 
$$T(G)=TT(G)\oplus TF(G)\cong K(G)\oplus\IZ/2\IZ\oplus\IZ\,,$$
so that it only remains to compute the group  $K(G)$ of trivial source endo-trivial modules. 
If $G/O_{2'}(G)\ncong \PGL_2^{\ast}(9)$, then $K(G)=X(G)$  by Proposition~\ref{prop:proofA(a)}, hence {\rm(a)}. 
If $G/O_{2'}(G)\cong \PGL_2^{\ast}(9)$, then the assertions in {\rm(b)} follow directly from Lemma~\ref{lem:SD5b}(a),(b) and (c).
\end{proof}

\bigskip
\bigskip
\bigskip

\section*{Acknowledgements}
\noindent
{The authors would like to thank Jesper Grodal, Burkhard K\"ulshammer, Gunter Malle,  Nadia Mazza, Benjamin Sambale,  Mandi Schaeffer Fry, Gernot Stroth, 
Jacques Th\'{e}venaz, and Rebecca Waldecker  
for helpful hints and discussions related to the content of this article. They are grateful  to the referee for detailed comments and suggestions.}
\bigskip


\bigskip
\bigskip

\bibliographystyle{amsalpha}
\bibliography{biblio.bib}

%


\end{document}